\documentclass[12pt]{amsart}
\usepackage{fullpage,url,amssymb,enumitem,colonequals}
\usepackage[alphabetic]{amsrefs}
\usepackage{lmodern}
\usepackage{relsize}
\usepackage{anyfontsize}
\setenumerate{listparindent=\parindent} 
\usepackage[all,cmtip]{xy} 
\usepackage{mathrsfs} 

\usepackage{tabularx} 
\usepackage{booktabs} 

\newcommand{\F}{\mathbb{F}}
\newcommand{\G}{\mathbb{G}}

\newcommand{\N}{\mathbb{N}}
\newcommand{\PP}{\mathbb{P}}

\newcommand{\Z}{\mathbb{Z}}

\newcommand{\calI}{\mathcal{I}}

\newcommand{\calO}{\mathcal{O}}

\DeclareMathOperator{\Aut}{Aut}

\DeclareMathOperator{\Char}{char}

\DeclareMathOperator{\Gal}{Gal}
\DeclareMathOperator{\Gr}{Gr}

\DeclareMathOperator{\im}{im}

\DeclareMathOperator{\Span}{Span}

\DeclareMathOperator{\Sym}{Sym}

\newcommand{\et}{{\operatorname{et}}}

\newcommand{\res}{{\operatorname{res}}}

\newcommand{\M}{\operatorname{M}}
\newcommand{\PGL}{\operatorname{PGL}}

\newcommand{\slantsf}[1]{\textsl{\textsf{#1}}}

\newtheorem{theorem}{Theorem}[section]
\newtheorem{lemma}[theorem]{Lemma}
\newtheorem{corollary}[theorem]{Corollary}
\newtheorem{proposition}[theorem]{Proposition}

\theoremstyle{definition}

\theoremstyle{remark}
\newtheorem{remark}[theorem]{Remark}

\usepackage{microtype}

\usepackage[
pdfauthor={Sachi Hashimoto and Borys Kadets}, 
]{hyperref}

\begin{document}

\title[38406501359372282063949 \& all that]{{\large 38406501359372282063949} \& all that: \\ {\small Monodromy of Fano Problems}}

\author{Sachi Hashimoto and Borys Kadets}
\thanks{BK was supported in part by National Science Foundation grant DMS-1601946 and Simons Foundation grant \#550033. SH was supported by a Clare Boothe Luce Fellowship (Henry Luce Foundation) and National Science Foundation grant DGE-1840990.}
\address{Sachi Hashimoto, Department of Mathematics and Statistics, Boston University, 111 Cummington Mall, Boston, MA 02215, USA}
\email{svh@bu.edu}
\urladdr{\url{http://math.bu.edu/people/svh/}}
\address{Borys Kadets, Mathematical Sciences Research Institute, Berkeley, CA, 94720-5070, USA}
\email{bkadets@msri.org}
\urladdr{\url{http://bkadets.github.io/}}

\begin{abstract}
A Fano problem is an enumerative problem of counting $r$-dimensional linear subspaces on a complete intersection in $\PP^n$ over a field of arbitrary characteristic, whenever the corresponding Fano scheme is finite. A classical example is enumerating lines on a cubic surface. We study the monodromy of finite Fano schemes $F_{r}(X)$ as the complete intersection $X$ varies. We prove that the monodromy group is either symmetric or alternating in most cases. In the exceptional cases, the monodromy group is one of the Weyl groups $W(E_6)$ or $W(D_k)$.
\end{abstract}

\maketitle

\section{Introduction}\label{S:introduction}
This paper studies the monodromy groups of Fano problems. Let $K$ be an algebraically closed field of any characteristic. By a \slantsf{Fano problem} we mean the problem of enumerating $r$-dimensional linear spaces on a general complete intersection $X_{[d]} = X_{d_1} \cap \dots \cap X_{d_s} \subset \PP^n$ of multidegree $[d]=(d_1, \dots, d_s)$, $d_i >1$, whenever this number is finite. Harris \cite{HarrisMonodromy} initiated the study of monodromy groups of enumerative problems; in particular, he computed the monodromy of lines on a cubic surface to be the Weyl group $W(E_6)$. The monodromy group reveals the underlying structure (or lack thereof) of the set of solutions to an enumerative problem: lines on a cubic surface have an intersection pairing that has to be preserved by the monodromy group.

 Given a Fano problem associated to the triple $([d],n,r)$, let $M_{[d]}\colonequals \prod_{i=1}^s \mathbb{P}^{\binom{d_i+n}{n}-1}$ be the space parameterizing tuples of hypersurfaces of type $[d]$.  Define the incidence scheme $I$ of $r$-planes in intersection of hypersurfaces by $I \colonequals \{((X_{d_i}), \Lambda) |  \bigcap_i X_{d_i} \supset \Lambda \} \subset   M_{[d]} \times \mathbb{G}(r,n).$ The projection $\pi\colon I \to M_{[d]}$ is generically \'{e}tale (see Theorem \ref{Debarre-Manivel}). The number of $r$-planes on a general complete intersection $X_{[d]}$ is $N([d])\colonequals \deg \pi$. Let $U \subset M_{[d]}$ be an open subscheme such that $\pi$ is finite \'{e}tale over $U$. Fix a basepoint $x \in U$ and let $\phi\colon \pi_1^{\et}(U,x) \to S_{N([d])}$ be the homomorphism to the symmetric group corresponding to the covering $\pi$. The image $G_{[d]} \colonequals \im \phi$ is the \slantsf{Fano monodromy group}. Equivalently, the Fano monodromy group is the Galois group of the extension of function fields $\Gal(K(I) /K(M_{[d]}))$.

Our main theorem is the following statement, which is a combination of Theorems \ref{mainslick} and \ref{mainhairy}. It states that, in most cases, Fano monodromy groups are large.

\begin{theorem}[Main Theorem]\label{intro-main}
Suppose we are not in the case of a cubic surface in $\PP^3$ or the intersection of two quadrics in $\PP^{n}$, so $([d],n,r) \neq ((3),3,1)$ and $[d]\neq (2,2)$. Then the Fano monodromy group is either the alternating group or the symmetric group: $G_{[d]} \supset A_{N\left([d]\right)}$.
\end{theorem}

The exceptional cases are the Fano monodromy group of  the cubic surface and the Fano monodromy groups for the intersections of two quadrics. These groups are forced to be small because the monodromy must preserve the 
intersection pairing.  For the case of a complete intersection of two quadrics, the monodromy group has to be a subgroup of the Weyl group $W(D_{2k+3})$ 
described in Section \ref{S:two-quadrics}. Nevertheless, we show that this monodromy group is as large as possible:
\begin{theorem}\label{2-quadrics}
	Suppose $\Char K \neq 2$. The Fano monodromy group of $k$-planes on a complete intersection of two quadrics in $\mathbb{P}^{2k+2}$ is the Weyl group $W(D_{2k+3})$.
\end{theorem}

The monodromy of lines on hypersurfaces of degree larger than $3$ is computed in \cite[Section~III.1]{HarrisMonodromy}: Harris shows that the relevant monodromy groups are symmetric groups. In particular, there is no structure to the set of lines on a hypersurface of degree larger than $3$.  A key ingredient in Harris' proof is a local computation of monodromy for a suitable degeneration of hypersurfaces. The main challenge in proving Theorems \ref{intro-main} and \ref{2-quadrics} is that for many degrees $[d]$ an analogous local computation fails:
we do not know how to find a complete intersection containing finitely many $r$-planes such that one of them is of multiplicity $2$.
Instead we use a purely global argument based on the classification of finite simple groups. In particular, we answer a question posed by Hassett and Tschinkel \cite[Question~6]{HassettTschinkelThreeQuadrics} by computing the monodromy in the case $([d],n,r)=((2,2,2),8,2)$.

While we expect the monodromy to be the symmetric group in all cases of Theorem \ref{intro-main}, our method is unable to distinguish between the alternating group and the symmetric group. Proving the groups are symmetric groups in all cases would require a genuinely new idea. 

Section \ref{Setup} outlines the strategy of the proof and sets up the notation. Section \ref{Numerology} collects several theorems about degrees and dimensions of Fano schemes of complete intersections of use in the rest of the paper. Sections \ref{double-transitivity} and \ref{sec:main} contain the bulk of the proof of Theorem \ref{intro-main}; Section \ref{S:ThreeQuadrics} covers the remaining cases. Finally, Section \ref{S:two-quadrics} contains the proof of Theorem \ref{2-quadrics}. 
\section{Setup and strategy of the proof}\label{Setup}
In this section we set up the notation for the rest of the paper.

Let \[M_{[d]} \colonequals \prod_{i=1}^s \PP^{\binom{d_i+n}{n}-1}\] denote the moduli space of tuples of hypersurfaces of degrees $[d]=(d_1, \dots d_s)$. The moduli space of tuples $(X_{d_i})_{i=1}^s$ that define a complete intersections of type $[d]$ is an open subscheme of $M_{[d]}$. We abuse notation by identifying an element $(X_{d_i})_{i=1}^s \in M_{[d]}$ with the intersection $X_{[d]}=\bigcap_{i=1}^s X_{d_i}$ (assuming that the complete intersection $X_{[d]}$ is equipped with the choice of defining equations).

Consider the incidence scheme $I \subset M_{[d]} \times \G(r,n)$ parameterizing pairs $(X_{[ d]}, \Lambda)$ with $\Lambda \subset X_{[d]}$. By Theorem \ref{Debarre-Manivel}, the projection $\pi\colon I \to M_{ [d]}$ is generically quasi-finite and generically smooth. Since $\pi$ is also proper, $\pi$ is generically \'{e}tale. Let $U \subset M_{ [d]}$ be an open subscheme over which $\pi$ is \'{e}tale. The Fano monodromy group $G_{[d]} \subset S_{N([d])}$ is the monodromy group of the covering $\pi\colon \pi^{-1}(U) \to U$. Since $M_{[d]}$ is normal, the Fano monodromy group is also the Galois group of the function field extension $K(I)/K(M_{[d]})$, and, in particular, does not depend on the choice of $U$.

 Recall that a \slantsf{permutation group} is a triple $(G, S, \varphi)$ consisting of a finite group $G$, a finite set $S$, and an action $\varphi: G \to \Aut(S)$ of $G$ on $S$. The number $\#S$ is the \slantsf{degree} of the permutation group. A permutation group $G$ is called \slantsf{$m$-transitive} if $m\leqslant\#S$ and for any two $m$-tuples of distinct points of $S$ there is an element of $G$ mapping one to the other. The \slantsf{transitivity degree} of $G$ is the largest $m$ for which $G$ is $m$-transitive. 

 The Fano monodromy group $G_{[d]}$ is a permutation group of degree $N([d]).$ We approach the problem of computing $G_{[d]}$ by studying the transitivity degree. The main group-theoretic input to the proofs of Theorems \ref{intro-main} and \ref{2-quadrics} is the following corollary of the classification of finite simple groups.

\begin{theorem}[see Sections 7.3 and 7.4 of \cite{Cameron1999}]\label{CFSG}
	Suppose $G$ is a $3$-transitive permutation group. Suppose also that one of the following holds:
	\begin{enumerate}
		\item  the degree $N$ of $G$ is not a power of $2$, $N \neq p^m+1$ for a prime $p$ and some $m \in \N$, and $N>24$,
		\item the group $G$ is $4$-transitive and $N>24$,	\end{enumerate} 
	then $G$ contains the alternating group $A_N$.
\end{theorem}
\begin{remark}
The number $24$ appears due to the existence of highly-transitive actions of Mathieu groups. 
\end{remark}
A finite \'{e}tale covering $f\colon X \to Y$ of an irreducible scheme $Y$ has transitive monodromy if and only if $X$ is connected, since the category of finite \'{e}tale coverings is equivalent to the category of $\pi_1^{\et}(X,x)$-sets; see \cite{SGA1}{\ Section V.7}. For example, to show that $G_{[d]}$ is transitive we will show that the second projection $\pi_2\colon I \to \G(r,n)$ is a proper dominant map with irreducible and equidimensional fibers. An $m$-transitive group is $(m+1)$-transitive if and only if the stabilizer of an $m$-tuple of points $T\subset S$ acts transitively on $S \setminus T$. Therefore, to prove higher transitivity we restrict the covering $\pi\colon I \to M_{[d]}$ to various closed subschemes $M' \subset M_{[d]}$, such that the monodromy of the restricted covering $\pi^{-1}(M') \to M'$ is contained in a stabilizer of a set of points.  In this way we reduce calculating transitivity degrees to describing irreducible components of the restricted coverings. 
\section{Numerology}\label{Numerology}

In this section we discuss the Fano scheme of a complete intersection in $\PP^n$ and its associated significant numbers. In particular, we present a formula of Debarre and Manivel \cite{DebarreManivel} for the expected dimension of the Fano scheme and discuss when it is zero. We also apply a (different) formula of Debarre and Manivel to obtain a lower bound on the number of $r$-planes in a complete intersection, when the Fano scheme is zero-dimensional.

Throughout, we work over an algebraically closed field $K$ of any characteristic. Let $X_{[d]} \colonequals X_{d_1} \cap \dots \cap X_{d_s}$ be the intersection of $s$ hypersurfaces $X_{d_i}$ in $\PP^n$, of degrees $[d] \colonequals (d_1, \dots, d_s)$ with $d_i>1$ for all $i$. The Fano scheme $F_r(X_{[ d]})$ is the moduli scheme of $r$-planes contained in $X_{[d]}$; it is a subscheme of the Grassmannian $\G(r,n)$. The expected dimension of $F_r(X_{[ d]})$ is \begin{align*}\delta = \delta({[d]}, n, r) &\colonequals (n - r)(r+1) -\binom{{[d]}+ r}{r}\\ &\ = (n-r)(r+1) - \sum_{i=1}^s \binom{d_i+r}{r}  .\end{align*} Let $\delta_{-} \colonequals \min\{ \delta, n - 2r - s\}$.
The following theorem is \cite[Theorem~2.1]{DebarreManivel}.
\begin{theorem}[Debarre-Manivel]\label{Debarre-Manivel}
\hfill
\begin{enumerate}
\item If $\delta_{-} < 0$, then for a general $X_{[d]}$ the Fano scheme $F_r(X_{ [d]})$ is empty. \label{Debarre-Manivel1}

\item If $\delta_{-} \geqslant 0$, then for a general $X_{[d]}$ the Fano scheme $F_r(X_{[d]})$ is smooth of dimension $\delta$.

\end{enumerate}

\end{theorem}

We are interested in the case when  $\dim F_r(X_{[d]}) =0$ for a general $X_{[d]}$. By Theorem \ref{Debarre-Manivel} this is equivalent to $\delta_{-} =\delta=0$. This implies the equality 
\begin{equation}\label{main-relation}
(n-r)(r+1)=\binom{[d]+r}{r}. \tag{$\star$}
\end{equation}

In the proof of Lemma \ref{NoCrossing} we use the following classification of complete intersections with $n-2r-s=0$. 
\begin{lemma}\label{DM-fanfiction}
Assume $\delta =0$. We have that $n - 2r -s  = 0$ if and only if either $s =1$, $d = 3$, and $r=1$, or $s= 2$ and $[d] = (2,2)$. 

\end{lemma}
\begin{proof}
As $\delta =0$ we can assume (\ref{main-relation}). We substitute $n - r = s+r$ into (\ref{main-relation}), and bound the right hand side from below by $s\binom{r+2}{r}$ to obtain $2 \geqslant s$. When $s = 1$ we obtain the equation \[(1+r)^2 = \binom{d +r}{r} \] from $\delta=0$. By considering the binomial coefficients that are perfect squares, we see by \cite{Erdos} that the only possibility is $r=1$ and $d = 3$. When $s = 2$, (\ref{main-relation}) becomes \[ (1+r)(2+r) = \binom{d_1 +r}{r} +  \binom{d_2 +r}{r}.\] The right side grows strictly monotonically in $d_i$, and is smallest when $d_1 = d_2 = 2$. On the other hand, when $d_1 = d_2 = 2$ we achieve equality.
\end{proof}
We would also like to determine the degree of the Fano scheme of a general complete intersection, which corresponds to the number of $r$-planes in $X_{[d]}$ when $\delta = 0$. For a single hypersurface $X_{d_i}$, the class of the Fano scheme $[F_r(X_{d_i})]$ in the Chow ring $A^{\binom{d_i+ r}{r}}( \G(r,n))$ is given by the top Chern class $c_{\binom{d_i+ r}{r}} \Sym^{d_i} S^\vee $, where $S$ is the tautological bundle on the Grassmannian. To compute the class of the complete intersection, we take the product $F_r(X_{[d]})=\prod_{i=1}^s[F_r(X_{d_i})]= \prod_{i=1}^{s} c_{\binom{d_i+ r}{r}} \Sym^{d_i} S^\vee. $ These Chern classes are readily computable using the splitting principle. Debarre and Manivel give a formula for the product of Chern classes using properties of symmetric functions.
\begin{theorem}[Theorem 4.3 of \cite{DebarreManivel}]
\label{DMthm2}
Suppose $X_{[d]}$ is a complete intersection with\\ $\dim F_r(X_{[d]})=\delta=0.$ 
Define \[Q_{r, d_i} \colonequals \prod_{a_0+ \dots + a_r = d_i} (a_0 x_0 + \dots + a_r x_r), \, \, a_i  \in \Z_{\geqslant 0},\] and $Q_{r, [d]} \colonequals \prod_{i=1}^s Q_{r, d_i}$. Let $V_r \colonequals \prod_{0 \leqslant i < j \leqslant r} (x_i- x_j)$ denote the Vandermonde polynomial. Then the degree of $F_r(X_{[d]})$ is the coefficient of the monomial $x_0^nx_1^{n-1}\dots x_r^{n-r}$ in $Q_{r, [d]} \cdot V_r.$
\end{theorem}
\begin{lemma}
\label{divisible}
When $F_r(X_{[d]})$ has dimension $\delta=0$, the degree of the Fano scheme $F_r(X_{[d]})$ is divisible by $\prod_{i=1}^s d_i^{r+1}$.
\end{lemma}
\begin{proof}
By Theorem \ref{DMthm2}, it suffices to prove that the coefficient of $x_0^nx_1^{n-1}\dots x_r^{n-r}$ in $Q_{r, {[d]}}$ is divisible by $\prod_{i=1}^s d_i^{r+1}$. For each $i = 1, \dots, s$ and $j=0, \dots, r$ the terms $d_i x_j$ divide $Q_{r, d_i}$ for all $j = 0, \dots, r$, and therefore divide $Q_{r, {[d]}}$. Collecting terms, we have that $\prod_{i=1}^s d_i^{r+1}$ divides $Q_{r, {[d]}}$.
\end{proof}
\begin{corollary}\label{NoMathieu}
Let $X_{[d]}$ be a complete intersection in $\PP^n$ such that $\dim F_r(X_{[d]})=\delta = 0$. Then $\deg (F_r(X_{[d]}))> 24$ except when $[d] = (2,2)$, $n = 4$, and $r = 1$.

\end{corollary}

\begin{proof}
For large enough $[d]$, we know $\prod_{i=1}^s d_i^{r+1} > 24$, and by the previous lemma we can conclude that the degree is greater than $24$. The remaining small cases are a subset of $d= (2)$ and $r=1, 2,$ or $ 3$, $d =  (2,2)$ and $r =1$, $d= (3)$ and $r = 1$, and $d =  (4)$ and $r=1$.

We rule out the remaining cases using (\ref{main-relation}). The case $d= (4), r=1$ as well as the cases $d= (2), r= 1, 2, 3$ are impossible by parity considerations. For $d= (3), r = 1$, we have the case of a cubic hypersurface in $\PP^3$ which has degree $27$. For $d =  (2,2),r =1$, we have the case of a degree $4$ del Pezzo in $\PP^4$. We compute $\deg F_1(X_{(2,2)}) = 16$.
\end{proof}

\section{Double transitivity}\label{double-transitivity}

\begin{proposition}\label{transitive}
	For any Fano problem $([d],n,r)$, the Fano monodromy group $G_{[d]}$ is transitive.
\end{proposition}
\begin{proof}
 	Let $U \subset M_{[d]}$ be an open subscheme such that $\pi$ is \'{e}tale over $U$.	We need to show that $\pi^{-1}(U)$ is connected. It suffices to show that $I$ is irreducible. Consider the projection $\pi_2\colon I \to \G(r,n)$. Given $\Lambda \in \G(r,n)$ the fiber $\pi_2^{-1}(\Lambda)$ is the moduli of intersections of hypersurfaces $X_{[d]}$ containing $\Lambda$. Given $\Lambda_1, \Lambda_2 \in \G(r,n)$, choose $g \in \PGL_{n+1} (K)$ such that $g\Lambda_1=\Lambda_2$. The automorphism $g$ induces an isomorphism of fibers $\pi_2^{-1}(\Lambda_1) \simeq \pi_2^{-1}(\Lambda_2)$. So $\pi_2$ is proper over an irreducible scheme, and has irreducible equidimensional fibers; thus the total space $I$ is irreducible.
\end{proof}

We will use an argument analogous to the proof of Proposition \ref{transitive} to show that $G_{[d]}$ is $4$-transitive for most $[d]$; combined with Theorem \ref{CFSG} this will quickly lead to the proof of Theorem \ref{mainslick}. Double transitivity is an input for proving $4$-transitivity, and we present it separately in Theorem \ref{2-trans}. We precede the proof of Theorem \ref{2-trans} with two lemmas that replace the use of the $\PGL_{n+1}(K)$-action in the proof of Proposition \ref{transitive}.

\begin{lemma}\label{dimension-count-pair}
	Suppose $\Lambda_1, \Lambda_2$ are linear subspaces of $\PP^n$, $\dim \Lambda_1= \dim \Lambda_2=r$, $\dim \Lambda_1 \cap \Lambda_2 = m$. Let $\calI_{\Lambda_1 \cup \Lambda_2}$ denote the ideal sheaf of $\Lambda_1 \cup \Lambda_2 \subset \PP^n$.
	Then for any $d\geqslant 2$ \[\dim H^0(\PP^n, \calI _{\Lambda_1 \cup \Lambda_2} (d) )= \binom{d+n}{n} - 2 \binom{d+r}{r} + \binom{d+m}{m}.\]
\end{lemma}
\begin{proof}
	The restriction map $H^0(\PP^n, \calO(d)) \to H^0(\Lambda_1, \calO(d))$ is surjective with kernel \\ $H^0(\PP^n, \calI_{\Lambda_1}(d))$, therefore $\dim H^0(\PP^n, \calI_{\Lambda_1}(d)) = \dim H^0(\PP^n, \calO(d)) - \dim H^0(\Lambda_1, \calO(d)) = \binom{n+d}{n}-\binom{r+d}{r}$. Consider the restriction map $\phi\colon H^0(\PP^n, \calI_{\Lambda_1}(d)) \to H^0(\Lambda_2, \calI_{\Lambda_1\cap \Lambda_2}(d))$. We want to show $\phi$ is surjective. Choose an $(n-r-1)$-dimensional subspace $\Sigma \subset \PP^n$ such that $\Sigma \cap \Lambda_2 = \emptyset$ and $\dim \Sigma \cap \Lambda_1 = r-m-1$.  Given an element  $F \in  H^0(\Lambda_2, \calI_{\Lambda_1 \cap \Lambda_2}(d))$, consider the hypersurface $Z \subset \Lambda_2$ given by the vanishing locus of the section $F$. Let $C \subset \PP^n$ be the cone with apex $\Sigma$ and basis $Z$ (i.e the union of all lines through $Z$ and $\Sigma$); so $\deg C=\deg F = d$. The cone $C$ contains $\Lambda_1$ and $C \cap \Lambda_2=Z$. Therefore we can choose an equation $\tilde{F}$ of $C$ that gives the desired lift of $F$ to $H^0(\PP^n, \calI_{\Lambda_1}(d)) $.
	
	Since $\phi$ is surjective, \begin{align*}\dim H^0(\PP^n, \calI _{\Lambda_1 \cup \Lambda_2} (d) ) = \dim \ker \phi &=   \dim H^0(\PP^n, \calI_{\Lambda_1}(d)) - \dim H^0(\Lambda_2, \calI_{\Lambda_1 \cap \Lambda_2}(d))\\ &=    \binom{d+n}{n} - 2 \binom{d+r}{r} + \binom{d+m}{m}. \end{align*}
\end{proof}
\begin{lemma}\label{NoCrossing}
	Suppose $\delta = 0$. Suppose $([d], n) \neq ((3), 3)$ and $[d] \neq (2,2)$. For a general complete intersection $X_{[d]}$ and arbitrary $\Lambda_1, \Lambda_2 \in F_r(X_{[d]}),$ with $\Lambda_1 \neq \Lambda_2$, the intersection $\Lambda_1 \cap \Lambda_2$ is empty.
\end{lemma}
\begin{proof}
	Let $I_m \subset M_{[d]} \times \G(r,n) \times \G(r, n)$ denote the scheme of triples $(X_{[d]}, \Lambda_1, \Lambda_2)$ such that $\Lambda_1, \Lambda_2 \subset X_{[d]}$ and $\dim \Lambda_1 \cap \Lambda_2 = m$. We need to show that the projection $I_m \to M_{[d]}$ is not dominant for any $0 \leqslant m \leqslant r-1$; we will do so by comparing the dimensions of $I_m$ and $M_{[d]}$.
	
	By definition
	$\dim M_{[d]} =  \binom{[d]+n}{ n} -s.$
	To compute the dimension of $I_m$, consider the projection $\pi_{23}\colon I_m \to \G(r,n) \times \G(r,n)$. The image $\pi_{23}(I_m)$ consists of pairs of subspaces $\Lambda_1, \Lambda_2$ with $\dim \Lambda_1 \cap \Lambda_2 =m$, so \[\dim \pi_{23}(I_m) = \dim \G(r,n) + \dim \G(m,r) + \dim \G(r-m-1, n-m-1).\]
	The fiber of $I_m \to \pi_{23}(I_m)$ above a point $(\Lambda_1, \Lambda_2)$ is a product of projective spaces \[\prod_{i=1}^s \left| H^0\!\left(\PP^n, \calI_{\Lambda_1 \cup \Lambda_2}(d_i)\right) \right|,\] whose dimension is given in Lemma \ref{dimension-count-pair}. Combining these two formulas gives
	\[\dim I_m = (2r-m+1)(n-r) + (m+1)(r-m) +  \left(\binom{[d]+n}{n} - 2 \binom{[d]+r}{r} + \binom{[d]+m}{m} -s \right).\]
	Define $f(m) \colonequals \dim I_m - \dim M_{[d]}$; using formula (\ref{main-relation}) we get \[f(m)=(2r-n-m)(m+1) + \sum_{i=1}^s \binom{d_i+m}{d_i}.\]
	Considered as a polynomial in $m$, $f$ has roots at $m=-1$ and $m=r$ by formula (\ref{main-relation}). Our goal is to show that $f(m)$ is negative for $0 \leqslant m \leqslant r-1$; we do so by showing that $f(0)<0$ and $f''(x) \geqslant 0$ for any $x\geqslant 0$.  At zero we get $f(0)=2r-n+s$, which is negative by Case \ref{Debarre-Manivel1} of Theorem \ref{Debarre-Manivel} and Lemma \ref{DM-fanfiction}. Note that the binomial coefficients $\binom{d+x}{d}$ are completely monotone as a function of $x$, i.e. all of the derivatives of $\binom{d+x}{d}$ are positive for $x \geqslant 0$. In particular $f'''(x) \geqslant 0$, for any $x \geqslant 0$. Therefore to show that $f''(x) \geqslant 0$ for all $x \geqslant 0$ it suffices to show that $f''(0)\geqslant 0$. By taking log derivatives, we see
	 \[\binom{d+x}{d}' = \binom{d+x}{d} \left(\frac{1}{1+x}+ \frac{1}{2+x} +\dots+ \frac{1}{d+x} \right).\] From this we compute
	 \[f''(0) = -2 + \sum_{i=1}^s\left( \left(1+\frac{1}{2} + ... \frac{1}{d_i}\right)^2 - \left(1 + \frac{1}{4} + ... + \frac{1}{d_i^2}\right)\right).\] This expression is easily seen to be negative if and only if $[d]=(2)$, in which case the dimension of the Fano scheme $\delta$ is nonzero.	
\end{proof}
\begin{theorem}\label{2-trans}
	Suppose $([d], n) \neq ((3), 3)$ and $[d] \neq (2,2)$. Then the Fano monodromy group $G_{[d]}$ is doubly transitive.
\end{theorem}
\begin{remark}
Theorem \ref{2-trans} detects all the special cases of Theorem \ref{intro-main}.
\end{remark}
We recall the general setup of the incidence variety and its two projection maps.

\centerline{
\xymatrix{\ar @{} [dr] 
I\ar[r]^-{ \mathlarger{ \pi}_2}\ar[d]_-{\mathlarger{\pi}} &\G(r,n)  \\\ M_{[d]}&  }
}
\begin{proof}
	Let $X_{[d]}$ be a sufficiently general complete intersection such that the covering $\pi$ is \'etale over $X_{[d]}$ and no two distinct $r$-planes in $F_r(X_{[d]})$ intersect; such $X_{[d]}$ exist by Lemma \ref{NoCrossing}. Let $\Lambda \in F_r(X_{[d]})$ be an $r$-plane. Consider the subscheme $M_\Lambda \subset M_{[d]}$ consisting of intersections of hypersurfaces that contain $\Lambda$. The scheme $M_\Lambda$ is a product of projective spaces, and in particular is irreducible. Since $X_{[d]} \in M_\Lambda$, the covering $\pi$ is generically \'{e}tale when restricted to $M_\Lambda$. Denote by $I_\Lambda$ the scheme $\pi^{-1}(M_\Lambda)$. The monodromy group $G_\Lambda$ of $I_\Lambda \to M_\Lambda$ is naturally a permutation subgroup of $G_{[d]}$. 
	
	The scheme $\pi_2^{-1}(\Lambda) \subset M_{[d]} \times \G(r,n)$ is a closed subscheme of $I_\Lambda$ and dominates $M_\Lambda$. Therefore, $\pi_2^{-1}(\Lambda)$ is an irreducible component of $I_\Lambda$; let $I_\Lambda^{\mathrm{irr}} \subset I_\Lambda$ denote the closure of the complement $I_\Lambda \setminus \pi_2^{-1}(\Lambda)$. Since the covering $I_\Lambda \to M_\Lambda$ decomposes into irreducible components,  the monodromy group $G_\Lambda$ is contained in a one-point-stabilizer of $G_{[d]}$. Since $G_{[d]}$ is transitive, to show that $G_{[d]}$ is $2$-transitive it suffices to show that $G_\Lambda$ acts transitively on a smooth fiber of $I_\Lambda^{\mathrm{irr}}  \to M_\Lambda$. 
	
	Let $U \subset \G(r,n)$ be the open subscheme $U \colonequals \{\Sigma \in \G(r,n)| \Sigma \cap \Lambda = \emptyset\}.$ Lemma \ref{NoCrossing} implies that for a general $Y_{[d]} \in M_\Lambda$ the fiber $\pi^{-1}(Y_{[d]})$ is contained in $\pi_2^{-1}(U)$. Therefore the projections $\pi_2^{-1}(U) \cap I_\Lambda^{\mathrm{irr}}  \to M_\Lambda$ and $I_\Lambda^{\mathrm{irr}}  \to M_\Lambda$ are isomorphic over the generic point; thus it suffices to show that the scheme $\widetilde{I}_\Lambda^{\mathrm{irr}}\colonequals \pi_2^{-1}(U) \cap I_\Lambda^{\mathrm{irr}} $ is irreducible.
	
	The projection $ \pi_2\colon \widetilde{I}_\Lambda^{\mathrm{irr}} \to U$ is proper. The fiber $\pi_2^{-1}(\Sigma) \cap \widetilde{I}_\Lambda^{\mathrm{irr}} $ over a point $\Sigma \in U$ consists of intersections of hypersurfaces of type $[d]$ that contain $\Lambda$ and $\Sigma$. The group $\mathrm{Stab}(\Lambda) \subset \PGL_{n+1}(K)$ acts transitively on the $r$-planes that do not intersect $\Lambda$, therefore the dimension $\dim \pi_2^{-1}(\Sigma) \cap \widetilde{I}_\Lambda^{\mathrm{irr}}$ does not depend on $\Sigma$. Since $\pi_2\colon \widetilde{I}_\Lambda^{\mathrm{irr}} \to U$ is proper with irreducible equidimensional fibers, the source $\widetilde{I}_\Lambda^{\mathrm{irr}}$ is irreducible.	
\end{proof}
\begin{remark}
	To prove $4$-transitivity we show linear independence of quadruples of linear subspaces on a sufficiently general complete intersection, similarly to Lemma \ref{NoCrossing}. At first glance, it might seem like an imposing task to go through all possible configurations of quadruples of linear subspaces. When $[d]\neq (2,2,2)$ we use an inductive approach. In the proof of Lemma \ref{bootstrap}, we deduce the linear independence of quadruples of subspaces from the linear independence of triples of subspaces. When $[d]=(2,2,2)$ the triples of subspaces are not linearly independent; this case is considered in Lemma \ref{no-malarkey}.
\end{remark}
 \section{Main Theorem}
 \label{sec:main}

We can now prove Theorem \ref{intro-main} in all cases except for $[d]=(2,2,2)$ with $n>8
$; this is stated in Theorem \ref{mainslick}. Theorem \ref{CFSG} combined with Corollary \ref{NoMathieu} reduces the problem to analyzing the transitivity degree of $G_{[d]}$, which is done inductively in Lemma \ref{bootstrap}.

\begin{lemma}\label{dimension-count-m}
	Let $m > 1$ be an integer. Suppose $\Lambda_1, ..., \Lambda_{m+1} \subset \PP^n$ are linear subspaces of dimension $r$ such that any $m$ of $\Lambda_i$'s are linearly independent (i.e. have maximal linear span). Then for any $d\geqslant 2$
	\[\dim H^0\left(\PP^n, \calI_{\Lambda_1 \cup ... \cup \Lambda_{m+1}}(d)\right)= \binom{d+n}{n} - (m+1) \binom{d+r}{r}.\]
\end{lemma}
\begin{proof}
	By induction on $m$ for $m>2$, and by Lemma \ref{dimension-count-pair} for $m = 2$, we can assume that \[\dim H^0\left(\PP^n, \calI_{\Lambda_1 \cup ... \cup \Lambda_m}(d)\right)= \binom{d+n}{n} - m \binom{d+r}{r}.\] 

	 Consider the restriction map $\res_{\Lambda_{m+1}}\colon H^0 \left(\PP^n, \calI_{\Lambda_1 \cup ... \cup \Lambda_m}(d)\right) \to H^0 \left(\Lambda_{m+1} , \calO(d)\right)$. By definition $\ker \res_{\Lambda_{m+1}} = H^0\left(\PP^n, \calI_{\Lambda_1 \cup ... \cup \Lambda_{m+1}}(d)\right),$ so the statement is equivalent to the surjectivity of $\res_{\Lambda_{m+1}}$ by rank-nullity. Products of linear polynomials form a generating set for $H^0 \left(\Lambda_{m+1} , \calO(d)\right)$, so it suffices to show that for any collection of $d$ hyperplanes $\Sigma_1, ..., \Sigma_d \subset \Lambda_{m+1}$ there exists a degree $d$ hypersurface $\Sigma \subset \PP^n$ such that $\Sigma\supset \Lambda_1 \cup ... \cup \Lambda_m$ and $\Sigma \cap \Lambda_{m+1}=\bigcup \Sigma_i$. Since $\Lambda_1, ..., \Lambda_{m-1}, \Lambda_{m+1}$ are linearly independent, there exists a hyperplane $H_1 \subset \PP^n$ such that $\Lambda_1 \cup ... \cup \Lambda_{m-1} \subset H_1$ and $H_1 \cap \Lambda_{m+1}=\Sigma_{1}$. Similarly let $H_2$ be a hyperplane through $\Lambda_m$ such that $H_2 \cap \Lambda_{m+1}=\Sigma_2$. Finally, choose hyperplanes $H_i \subset \PP^n$, $i=3, ..., d$ such that $H_i \cap \Lambda_{m+1}=\Sigma_i$. The hypersurface $\Sigma \colonequals H_1 \cup ... \cup H_d$ is the desired lift of $\bigcup_i \Sigma_i$ to $H^0 \left(\PP^n, \calI_{\Lambda_1 \cup ... \cup \Lambda_m}(d)\right).$
\end{proof}

\begin{lemma}\label{bootstrap}
	Let $m>1$ be an integer such that $n+1\geqslant m(r+1)$ and $\binom{d_i+n}{n} > (m+1) \binom{d_i+r}{r}$ for all $d_i$. Suppose that $G_{[d]}$ is $m$-transitive and that any $m$ distinct $r$-planes on a general complete intersection of type $[d]$ are linearly independent. Then $G_{[d]}$ is $(m+1)$-transitive. If in addition $n+1 \geqslant (m+1)(r+1)$, then any $m+1$ distinct $r$-planes on a general complete intersection are linearly independent.
\end{lemma}
\begin{proof}
	Let $\Lambda_1, ..., \Lambda_m \subset \PP^n$ be linearly independent $r$-planes; they exist since $n+1\geqslant m(r+1)$. Let $X_{[d]}$ be a sufficiently general complete intersection such that $X_{[d]} \supset \Lambda_i$ for all $i$ and $\pi\colon I \to M_{[d]}$ is \'etale above $X_{[d]}$. 
	
	Let $M_\Lambda \subset M_{[d]}$ denote the subscheme of intersections of hypersurfaces that contain $\Lambda_1 \cup ... \cup \Lambda_m$. Since $X_{[d]} \in M_\Lambda$, the covering $\pi$ is generically \'{e}tale when restricted to $M_\Lambda$. 
	
	Note that for any $i$, the subscheme $\pi^{-1}(M_\Lambda) \cap \pi_2^{-1}(\Lambda_i) \subset I$ maps isomorphically to $M_\Lambda$. Let $I_\Lambda^{\mathrm{irr}} \subset I$ denote the union of irreducible components of $\pi^{-1}(M_\Lambda)$ that are not of the form $\pi^{-1}(M_\Lambda) \cap \pi_2^{-1}(\Lambda_i)$. The covering $I_\Lambda^{\mathrm{irr}} \to M_\Lambda$ is generically \'{e}tale of degree $N([d])-m$. The monodromy of $I_\Lambda^{\mathrm{irr}} \to M_\Lambda$ is contained in the pointwise stabilizer of an $m$-tuple of points in $G_{[d]}$. Therefore to show that $G_{[d]}$ is $(m+1)$-transitive it suffices to show that the monodromy of $I_\Lambda^{\mathrm{irr}} \to M_\Lambda$ is transitive.

	Consider the open subscheme $U \subset \G(r,n)$ consisting of $r$-planes $\Lambda$ such that $\Lambda$ and an arbitrary set of $m-1$ distinct $\Lambda_i$ form a set of $m$ linearly independent spaces. Let $\widetilde{I}_\Lambda^{\mathrm{irr}}  \colonequals \pi_2^{-1}(U) \cap I_\Lambda^{\mathrm{irr}}$. By assumption, the preimage of $X_{[d]}$ under $\pi\colon I_\Lambda^{\mathrm{irr}} \to M_\Lambda$ is contained in $\widetilde{I}_\Lambda^{\mathrm{irr}}$, therefore the monodromy groups  of $\widetilde{I}_\Lambda^{\mathrm{irr}} \to M_\Lambda$ and of $I_\Lambda^{\mathrm{irr}} \to M_\Lambda$ are equal. Hence it is enough to show that the monodromy of $\widetilde{I}_\Lambda^{\mathrm{irr}} \to M_\Lambda $ is transitive; we will do so by showing that $\widetilde{I}_\Lambda^{\mathrm{irr}}$ is irreducible.
	
	Consider the projection $\pi_2\colon \widetilde{I}_\Lambda^{\mathrm{irr}} \to U$; it is surjective, since, for all $i$, we assume $\binom{d_i+n}{n}>(m+1)\binom{d_i+r}{r}$, and so for any $m+1$ distinct $r$-planes there exists a nontrivial intersection of hypersurfaces of type $[d]$ containing them. The fiber $\pi_2^{-1}(\Lambda) \cap \widetilde{I}_\Lambda^{\mathrm{irr}}$ has dimension $\sum_{i=1}^s \left(\binom{d_i+n}{n} - (m+1)\binom{d_i+r}{r}-1\right)$ by Lemma \ref{dimension-count-m}. Therefore $\pi_2\colon \widetilde{I}_\Lambda^{\mathrm{irr}} \to U$ is a proper dominant map with irreducible equidimensional fibers, so $\widetilde{I}_\Lambda^{\mathrm{irr}}$ is irreducible.
	
	Finally, we need to show that any $m+1$ distinct $r$-planes on a general complete intersection are linearly independent. Since $G_{[d]}$ is $(m+1)$-transitive, it suffices to show that a general complete intersection contains an $(m+1)$-tuple of linearly independent $r$-planes. Let $U_\Lambda \subset \G(r,n)$ denote the set of planes $\Lambda$, such that the subspaces $\Lambda, \Lambda_1, ..., \Lambda_m$ are linearly independent; $U_\Lambda$ is nonempty by the assumption $n+1\geqslant (m+1)(r+1)$. Since the projection $\widetilde{I}_\Lambda^{\mathrm{irr}} \to U$ is surjective and $\widetilde{I}_\Lambda^{\mathrm{irr}}$ is irreducible, the set $\pi_2^{-1}(U_\Lambda) \cap \widetilde{I}_\Lambda^{\mathrm{irr}}$ is a dense open subset of $\widetilde{I}_\Lambda^{\mathrm{irr}}$. Hence there is an open subset $V \subset M_{\Lambda}$ such that for every $X_{[d]}\in V$, the scheme $X_{[d]}$ contains $m+1$ linearly independent subspaces.
	
	Consider the morphism $\phi\colon \PGL_{n+1} \times M_\Lambda \to M_{[d]}$ induced by the linear action of $\PGL_{n+1}$ on $\PP^n$. Since a general complete intersection contains an $m$-tuple of linearly independent $r$-planes, the morphism $\phi$ is dominant. Thus the image $\phi(\PGL_{n+1} \times V)$ contains an open subset of $M_{[d]}$, so a general complete intersection contains $m+1$ linearly independent $r$-planes.
\end{proof}

Lemma \ref{Large-enough-for-4} checks when the conditions of Lemma \ref{bootstrap} are satisfied with $m = 3$.

\begin{lemma}\label{Large-enough-for-4}
Suppose $([d],n,r)$ is not one of the following exceptional Fano problems:
\begin{itemize}
\item $[d] = (2,2)$, $n = 2r+2$;
\item $[d] = (2,2,2)$, $n > 8$, $n = 5r+3$;
\item $[d] = (3)$, $n= 3, 8, 11, 18, 22$, $r =1,3, 4, 6, 7 $;
\item $[d]=(4), n = 7, r = 2$;
\item $[d]= (5), n = 4, r = 1$.
\end{itemize}
Then $n+1\geqslant 3(r+1)$ and $\binom{d_i+n}{n} > 4\binom{d_i+r}{r}$ for all $d_i$.
\end{lemma}

\begin{proof}
If $n > 4 r$, then by the convexity of the binomial coefficient, we obtain the bound 
\[\binom{d_i+n}{n} > \binom{d_i+4r}{4r} > 4\binom{d_i+r}{r},\] so the conclusion of the lemma holds.  When $r > 2$ or $r \leqslant 2$ and $n>9$, the condition $n > 4r$ ensures $n + 1\geqslant 3 (r+1)$.

In the cases $r = 1$ and $2$ and $n \leqslant9$, a straightforward computation gives all of the finitely many Fano problems $([d],n, r)$; all of them but $((3),3,1)$, $((5),4,1)$, $((2,2),4,1),$ $((2,2), 6, 2),$ and $((4),7,2)$ satisfy the conclusion of the lemma. The Fano problems with $r = 1, 2$ and $n \leqslant 9$ are available at the website \cite{Github}.

We now find all triples $([d], n, r)$ such that $n \leqslant 4r$.

Suppose $[d] \neq (3)$.  Then either $X_{[d]}$ is an intersection of at least three quadrics, or $d_i>2$ for some $i$. In the all quadrics case, first suppose $s \geqslant 4$. Then (\ref{main-relation}) implies $(n-r)(r+1) = s \binom{2+r}{r} \geqslant 4 \binom{2+r}{r}$ and, dividing by $r+1$, we obtain $n + 1 \geqslant 3r +3 \geqslant 3 ( r+1)$.  Then we check that $\binom{2 + n }{n}> 4\binom{2 + r}{r}$: this follows from substituting $n \geqslant 3r + 4$. When $s= 3$, and $n = 8$ the inequalities hold as well.

Now suppose $d_i>2$ for some $i$. The equality (\ref{main-relation}) and the inequality $n \leqslant 4r$ imply $(3r)(r+1) \geqslant (n-r)(r+1) \geqslant \binom{3+r}{r}$. So we obtain the condition $r^2 - 13r + 6 \leqslant 0$, which holds if and only if $r \leqslant 12$. 
If $d_i>12$ for some $i$, then from the equation (\ref{main-relation}) we get \[n-r=\binom{[d]+r}{r} / (r+1) > \binom{12+r}{r}/(r+1) \geqslant 3r.\]
 It remains to find all triples $\left([d], n, r\right)$ with $r \leqslant 12, d_i \leqslant 12$, $\delta = 0$, and $d_i>2$ for some $i$.  By another computer search, the remaining cases are given in Table \ref{smallfanotable}.

	\begin{table}[h]
	\begin{center}
		\begin{tabularx}{195pt}{l l l l l} \toprule
		$[d]$  & $n$ & $r$  & $\binom{d_i+n}{n}$ & $\binom{d_i+r}{r}$\\
		\midrule
		3  & 3 & 1 & 20 & 4\\
				\midrule
		3  & 8 & 3 & 165 & 20\\
				\midrule
		3  & 11 & 4  & 364 & 35\\
				\midrule
		3  & 18 & 6  & 1330 & 84\\
			\midrule
		3  & 22 & 7  & 2300 & 120 \\
			\midrule
		3  & 31 & 9 & 5984  & 220 \\
			\midrule
		3 & 36 & 10 & 9139 & 286\\
			\midrule
		3  & 47 & 12 & 19600 & 455\\
			\midrule
		4  & 7 & 2 & 330  & 15\\
			\midrule
		5  &  4 & 1 & 126 & 6\\
		\midrule 
		$(2,3)$  & 14 & 4 & $(120, 680)$& $(15, 35)$ \\
			\midrule
		$(2,3)$  & 22 & 6  &$(276, 2300)$ & $(28, 84)$\\
			\bottomrule	
		\end{tabularx}	
		\caption{List of Fano problems for Lemma \ref{Large-enough-for-4} with small $r$ and $d$}\label{smallfanotable}
	\end{center}

\end{table}

We can verify directly that $n+1 \geqslant 3(r+1)$ and $\binom{d_i+n}{n} > 4\binom{d_i+r}{r}$ for all $d_i \in [d]$ in all cases except the cubic hypersurfaces where $r= 1, 3,4,6,$ or $7$ and the quartic surface with $r = 2$. 
\end{proof}

In the cases $[d]=(3)$ and $r \leqslant 7$, $[d]=(4)$ and $r = 2$, and $[d]=5$ and $r = 1$ not covered by Lemma \ref{Large-enough-for-4}, we note that $n +1 \geqslant 2(r+1)$ and $\binom{d_i+n}{n} > 3\binom{d_i+r}{r}$, so Lemma \ref{bootstrap} holds with $m = 2$. Therefore we will be able to conclude that the Fano monodromy group is $3$-transitive. Since we cannot show it is $4$-transitive using Lemma \ref{bootstrap}, we will apply a different case of Theorem \ref{CFSG}: if the degree $N$ of $G$ is not a power of $2$, and $N \neq p^m+1$ for a prime $p$ and some $m \in \N$, and $N>24$, then $G$ contains the alternating group. To do so, we must compute the degrees of the Fano schemes.

\begin{lemma}
\label{38406501359372282063949}
Let $\delta= 0.$ When $d= (3)$, $r \leqslant 6$, $d=(4)$ and $r = 2$, or $d=(5)$ and $r=1$ the degrees of the Fano schemes are 

\begin{itemize}
\item $[F_1(X_3)] = 27$,
\item $[F_3(X_3)] = 321489,$
\item $[F_4(X_3)] = 1812646836,$
\item $[F_6(X_3)] = 38406501359372282063949,$
\item $[F_2(X_4)] = 3297280,$
\item $[F_1(X_5)]= 2875$
\end{itemize}
and these degrees are not powers of two nor one more than a power of a prime.
\end{lemma}

\begin{proof}
These degrees can be computed directly using Theorem \ref{DMthm2}. Some care is needed to compute $[F_6(X_3)]$, as a naive implementation of the formula consumes too much memory. The polynomial $Q_{6,3}\cdot V_6$ contains $105$ linear factors; after multiplying together $50$ terms in $Q_{6,3}$ we have $1002496$ monomials, and after another ten terms we reach $10549495$, roughly a factor of ten more. To avoid this memory problem, we converted the multivariate polynomial $Q_{6,3}(x_0, \dots, x_6)$ to a univariate polynomial $Q(t)$ using a version of Kronecker substitution: \[Q(t) \colonequals Q_{6,3}(1, t, t^e, t^{e^2}, t^{e^3}, t^{e^4}, t^{e^5})\] for some $e \in \N$ such that the monomial $t^{17+16e+15e^2+14e^3+13e^4+12e^5}$ is represented by a unique monomial in $x_1, ..., x_6$ under the substitution $x_i=t^{e^{(i-1)}}$. We claim that $e=20$ works (in fact, this is the smallest $e$ that does). That is: we show that there is no septuple $(a_0, \dots, a_6) \in \N^7$, $(a_0, ..., a_6)\neq (18, ..., 12)$ with $\sum_{i=0}^6 a_i = 105$ and \[a_1+a_2e+a_3e^2+a_4e^3+a_5e^4+a_6e^5 = 17+16e+15e^2+14e^3+13e^4+12e^5.\] Since the right hand side is a representation of a number in base $e$ arithmetic, we have $\sum_{i=1}^6 a_i \geqslant 17+...+12=87$, with equality only in the case $(a_1, ..., a_{6})=(17, ..., 12)$. Reducing modulo $(e-1)$ gives $\sum_{i=1}^6 a_i = 87 \pmod {e-1}$. Since $87 \leqslant \sum_{i=1}^6 a_i \leqslant 105 < 87+e-1$, the congruence implies $\sum_{i=1}^6 a_i=87$ . 

 On the univariate side, many terms will collide, providing a savings in memory, and we can extract the coefficient of the desired monomial.
\end{proof}
\begin{lemma}
\label{infeasibledeg}
The degree of $F_7(X_3)$ is not a power of two or one more than a power of a prime. 
\end{lemma}

\begin{proof}
Using the method described in Lemma \ref{38406501359372282063949}, we compute $[F_7(X_3)] \equiv 3 \mod 4$.
\end{proof}

We can now prove the main theorem.

\begin{theorem}\label{mainslick}
	Suppose $([d],n,r) \neq ((3),3,1)$, $[d]\neq (2,2)$, and if $n>8$ suppose also that $[d] \neq (2,2,2)$. Then $A_{N\left([d]\right)} \subset G_{[d]}$.
\end{theorem}

\begin{proof}
	By Theorem \ref{CFSG} and Corollary \ref{NoMathieu} it suffices to show that $G_{[d]}$ is $4$-transitive. By Theorem \ref{2-trans}, the group $G_{[d]}$ is $2$-transitive. Applying Lemma \ref{bootstrap} repeatedly with $m=2$ and $m=3$, and using the estimates of Lemma \ref{Large-enough-for-4} we conclude that $G_{[d]}$ is $4$-transitive unless $ [d]= (3)$ and $r \leqslant 7$, $[d]=(4)$ and $r = 2$, or $[d]=(5)$ and $r = 1$. In those cases, by Theorem \ref{CFSG} it suffices to show that $G_{[d]}$ is $3$-transitive and the degree $N([d])$ is not a power of $2$ or one more than a power of a prime. Lemma \ref{bootstrap} with $m = 2$ and Lemma \ref{Large-enough-for-4} show $G_{[d]}$ is 3-transitive; Lemmas \ref{38406501359372282063949} and \ref{infeasibledeg} compute the relevant degrees.
\end{proof}

In the exceptional cases $([d],n,r)=((3),3,1)$ and $[d]=(2,2)$ there are restrictions on the monodromy imposed by the intersection pairings. However, the exceptional case $[d]=(2,2,2)$ is an artifact of the proof. This case can be dealt with by a similar method, as we now show.

\section{Intersections of three quadrics}\label{S:ThreeQuadrics}

The Fano monodromy group $G_{[d]}=G_{(2,2,2)}$ turns out to be large, however the proof of Theorem \ref{mainslick} needs to be modified to cover this case. For $[d]=(2,2,2)$ equation (\ref{main-relation}) implies $r=2k, n=5k+3$ for a positive integer $k$. In particular Lemma \ref{bootstrap} is not strong enough to show $4$-transitivity of $G_{[d]}$ when $k>1$. Instead in Lemmas \ref{one-two-three's}, \ref{3-quadrics-dim-estimate}, and \ref{no-malarkey} we carefully study configurations of triples of linear spaces lying on intersections of a triple of quadrics.  

\begin{lemma}\label{one-two-three's}
	Given $\Lambda_1, \Lambda_2, \Lambda_3 \subset \PP^{5k+3}$ with $\dim \Lambda_i=2k$ for all $i$, define \[\Lambda_{123}\colonequals \Span (\Lambda_1, \Lambda_2) \cap \Span(\Lambda_2, \Lambda_3) \cap \Span(\Lambda_3, \Lambda_1).\] Then for a general choice of $\Lambda_1, \Lambda_2, \Lambda_3$ the equality $\dim \Lambda_{123}=2k-3$ holds.
\end{lemma}
\begin{proof}
	Let $\Lambda_i^\perp\subset (\PP^{5k+3})^*$ denote the projective dual of $\Lambda_i$, $\dim \Lambda_i^\perp = 3k+2$. The dual of $\Lambda_{123}$ is given by \begin{multline*}
	\Lambda_{123}^\perp\colonequals\left( \Span (\Lambda_1, \Lambda_2) \cap \Span(\Lambda_2, \Lambda_3) \cap \Span(\Lambda_3, \Lambda_1)\right)^\perp = \\ \Span\left(\Lambda_1^\perp \cap \Lambda_2^\perp, \Lambda_2^\perp\cap\Lambda_3^\perp, \Lambda_3^\perp \cap \Lambda_1^\perp \right).\end{multline*}
	For a general choice of $\Lambda_1^{\perp}, \Lambda_2^{\perp}, \Lambda_3^\perp$ the intersections $\Lambda_i^\perp\cap \Lambda_j^\perp$ are linearly independent subspaces of dimension $k+1$. Therefore for a general choice of $\Lambda_1, \Lambda_2, \Lambda_3$ we have:
	$\dim \Lambda_{123} = 2k-3.$
\end{proof}

\begin{lemma}\label{3-quadrics-dim-estimate}
	Suppose $\Lambda_1, \Lambda_2, \Lambda_3 \subset \PP^{5k+3}$ are linear subspaces of dimension $2k$ such that $\Lambda_i \cap \Lambda_j = \emptyset$ for $i \neq j$ and the space $\Lambda_{123}$ defined in Lemma \ref{one-two-three's} is of dimension $2k-3$. Let $\Lambda \subset \PP^{5k+3}$ be a linear space of dimension $2k$ such that $\dim \Lambda \cap \Lambda_{123} = m$. Then 
	\[\dim H^0\left(\PP^n, \calI_{\Lambda\cup\Lambda_1\cup\Lambda_2\cup\Lambda_3}(2)\right) \leqslant \binom{5k+5}{2}-4\binom{2k+2}{2}+\binom{m+2}{2} .\] If $\Lambda \cap \Lambda_{123}=\emptyset$, then the equality 
	\[\dim H^0\left(\PP^n, \calI_{\Lambda\cup\Lambda_1\cup\Lambda_2\cup\Lambda_3}(2)\right) = \binom{5k+5}{2}-4\binom{2k+2}{2}\] holds.
\end{lemma}
\begin{proof}
	By Lemma \ref{dimension-count-m} the space $H^0(\PP^n, \calI_{\Lambda_1 \cup \Lambda_2 \cup \Lambda_3}(2))$ of quadrics vanishing on $\Lambda_1, \Lambda_2, \Lambda_3$ has dimension $\binom{5k+5}{2}-3\binom{2k+2}{2}$. Consider the restriction map \[\phi\colon H^0\left(\PP^n, \calI_{\Lambda_1\cup\Lambda_2\cup\Lambda_3}(2)\right) \to H^0\left(\Lambda, \calO(2)\right). \] The kernel $\ker \phi$ is equal to $H^0\left(\PP^n, \calI_{\Lambda\cup\Lambda_1\cup\Lambda_2\cup\Lambda_3}(2)\right)$, therefore to prove the lemma it suffices to show that $\dim \im \phi \geqslant \binom{2k+2}{2} - \binom{m+2}{2}.$ For the purposes of this proof, we use the convention $\dim \emptyset = -1$, so that when $m=-1$ we get $\dim \im \phi \geqslant \binom{2k+2}{2}$. Therefore, if $\Lambda \cap \Lambda_{123} = \emptyset$, surjectivity of $\phi$ implies $\dim H^0\left(\PP^n, \calI_{\Lambda\cup\Lambda_1\cup\Lambda_2\cup\Lambda_3}(2)\right) = \binom{5k+5}{2}-4\binom{2k+2}{2}$. 
	
	To bound the image of $\phi$ from below, we explicitly construct quadrics on $\Lambda$ that can be lifted to $ H^0\left(\PP^n, \calI_{\Lambda_1\cup\Lambda_2\cup\Lambda_3}(2)\right)$. For $i \neq j$ let $\Lambda_{ij}$ denote the space $\Span(\Lambda_i, \Lambda_j) \cap \Lambda.$ Suppose $H_{12}, H_3 \subset \Lambda$ are hyperplanes such that $\Lambda_{12} \subset H_{12}$. Then there is a lift of $H_{12}$ to a hyperplane $H_{12}' \subset \PP^n$ such that $H_{12}'\cap \Lambda= H_{12}$, $H_{12}' \supset \Span(\Lambda_1, \Lambda_2)$, and a lift of $H_3$ to a hyperplane $H_3'$ that contains $\Lambda_3.$ Therefore $H_{12}H_3=\phi(H_{12}'H_3') \in \im \phi$. By symmetry, the image of $\phi$ contains products of hyperplanes that contain $\Lambda_{ij}$ for some $i,j$. Let $V$ denote the space \[ V \colonequals \Lambda_{12} \cap \Lambda_{13} \cap \Lambda_{23}=\Lambda \cap \Lambda_{123}.\] By assumption $\dim V=m.$ 
	By the definition of $V$, taking duals in $\Lambda$, $V^\perp = \Lambda_{12}^\perp \oplus \Lambda_{13}^\perp \oplus \Lambda_{23}^\perp. $
	So we can choose a coordinate basis $x_1, ..., x_{2k+1} \in H^0\left(\Lambda, \calO(1)\right) $ such that 
	\begin{enumerate}
	\item $x_1,..., x_{2k-\dim\Lambda_{12}}$ vanish on $\Lambda_{12}$;
	\item $x_{2k-\dim\Lambda_{12}+1}, ..., x_{2k-\dim\left( \Lambda_{12} \cap \Lambda_{23}\right)}$ vanish on $\Lambda_{23}$;
	\item $ x_{2k-\dim\left( \Lambda_{12} \cap \Lambda_{23}\right)+1}, ..., x_{2k-m}$ vanish on $\Lambda_{13}$.
	\end{enumerate}
The quadrics $x_ix_j=0$ where $j\geqslant i$ and $i\leqslant 2k-m$ are in the image of $\phi$; they span the space of quadrics that contain $V$. Therefore $\dim (\im \phi) \geqslant \dim H^0\left(\Lambda, \calI_{V}(2)\right)=\binom{2k+2}{2}-\binom{m+2}{2}.$  
\end{proof}

\begin{lemma}\label{no-malarkey}
	Let $\Lambda_1, \Lambda_2, \Lambda_3 \subset \PP^{5k+3}$ be linear subspaces of dimension $2k$ such that $\Lambda_i \cap \Lambda_j = \emptyset$ for $i \neq j$, and such that the space $\Lambda_{123}$ defined in Lemma \ref{one-two-three's} has dimension $2k-3$. Consider the moduli space $M$ of triples of quadrics that contain $\Lambda_1, \Lambda_2,$ and $\Lambda_3$. Then for a general complete intersection $X \in M$ and an arbitrary $2k$-plane $\Lambda \subset X$, $\Lambda \neq \Lambda_1, \Lambda_2, \Lambda_3,$ the intersection $\Lambda \cap \Lambda_{123}$ is empty.
\end{lemma}
\begin{proof}
	Let $I_m \subset M \times \G(2k,5k+3)$ denote the space of pairs $(Q, \Lambda)$ with $Q \in M$ and $\Lambda \subset Q$ a $2k$-plane with $\dim \Lambda \cap \Lambda_{123}=m$. Our goal is to show that the images of the projections $I_m \to M$ for $m\geqslant 0$ do not cover the generic point of $M$. We do so by proving that $\dim I_m < \dim M$ for $0 \leqslant m \leqslant 2k - 3$.
	
	Let $\Gr_m$ denote the moduli space of $2k$-planes $\Lambda$ such that $\dim \Lambda \cap \Lambda_{123}=m$. Then 
	\begin{multline*}
			\dim \Gr_m= \dim \G(m, 2k-3) + \dim \G(2k-m-1, 5k+2-m)\\
			= (m+1)(2k-3-m) + (2k-m)(3k+3).
	\end{multline*}
	 By Lemma \ref{3-quadrics-dim-estimate} the fibers of the projection $I_m \to \Gr_m$ have dimension at most  \\ $3\left( \binom{5k+5}{2}-4\binom{2k+2}{2}+\binom{m+2}{2}\right)$, therefore 
	 \begin{equation}
	 \label{dimIm}
	 \dim I_m \leqslant (m+1)(2k-3-m) + (2k-m)(3k+3) + 3\left(\binom{5k+5}{2}-4\binom{2k+2}{2}+\binom{m+2}{2}\right).
	 \end{equation}
	  By Lemma \ref{dimension-count-m}, we have $\dim M = 3 (\binom{5k+5}{2}-3\binom{2k+2}{2}).$ Define $f(m) \colonequals \dim I_m - \dim M$. Then by subtracting the formula for $\dim M$ from (\ref{dimIm}) and simplifying, we obtain the bound $f(m) \leqslant (m+1)\left(-k-3+m/2\right).$ Since $f(m)<0$ for $0 \leqslant m \leqslant 2k-3$, none of the projections $I_m \to M$ are dominant.	
\end{proof}

\begin{theorem}\label{mainhairy}
	Suppose $[d]=(2,2,2)$ and $n>8$. Then $A_{N([d])} \subset G_{[d]}.$
\end{theorem}
\begin{proof}
 Let $I, M_{[d]}, \pi$ be as in the setup, so that $G_{[d]}$ is the monodromy group of $\pi\colon I \to M_{[d]}$. Let $X \in M_{[d]}$ be a sufficiently general complete intersection, and let  $\Lambda_1, \Lambda_2, \Lambda_3 \subset X$ be $2k$-planes such that the following conditions are satisfied:
 \begin{enumerate}
 	\item $\pi$ is \'{e}tale over $X$,
 	\item\label{this-is-2}  $\Lambda_i \cap \Lambda_j = \emptyset$ for $i \neq j$,
 	\item\label{this-is-3}  $\Lambda_{123}\colonequals \Span (\Lambda_1, \Lambda_2) \cap \Span(\Lambda_2, \Lambda_3) \cap \Span(\Lambda_3, \Lambda_1)$ has dimension $2k-3$.
 \end{enumerate}

Conditions \ref{this-is-2} and \ref{this-is-3} hold for a general triple of $2k$-planes in $\PP^{5k+3}$ by Lemma \ref{one-two-three's}. Since there exists a quadric containing an arbitrary triple of $2k$-planes in $\PP^{5k+3}$, conditions \ref{this-is-2} and \ref{this-is-3} are satisfied for a general $X$. Let $M_\Lambda \subset M_{[d]}$ be the moduli space of quadrics containing $\Lambda_1, \Lambda_2$, and $\Lambda_3$. Since $X \in \M_\Lambda$, the covering $\pi$ restricted to $M_\Lambda$ is generically \'{e}tale. We now study the transitivity properties of the monodromy group of the covering $I_\Lambda \colonequals \pi^{-1}(M_\Lambda) \cap I \to M_\Lambda$ to show $4$-transitivity of $G_{[d]}$.

The schemes $\pi_2^{-1}(\Lambda_i) \cap I_\Lambda$ are irreducible components of $I_\Lambda $ that map isomorphically to $M_\Lambda$ by $\pi$. Let $I_\Lambda^{\mathrm{irr}}$ denote the union of irreducible components of $I_\Lambda \setminus \bigcup \left(\pi_2^{-1}(\Lambda_i) \cap I_\Lambda\right).$ The monodromy group of $I_\Lambda \to M_\Lambda$ is a subgroup of the pointwise stabilizer of a triple of points in $G_{[d]}$. To show that $G_{[d]}$ is $4$-transitive it suffices to show that the monodromy of $I_\Lambda^{\mathrm{irr}} \to M_\Lambda$ is transitive.
Let $U \subset \G(2k,n)$ be the open subscheme of subspaces $\Lambda$ such that $\Lambda \cap \Lambda_{123} = \emptyset$. Define $\widetilde{I}_\Lambda^{\mathrm{irr}}\colonequals \pi_2^{-1}(U) \cap I_\Lambda^{\mathrm{irr}}$. By Lemma \ref{no-malarkey} the coverings $I_\Lambda^{\mathrm{irr}} \to M_\Lambda$ and $\widetilde{I}_\Lambda^{\mathrm{irr}} \to M_\Lambda$ are isomorphic over the generic point of $M_\Lambda$. Therefore it suffices to show that $\widetilde{I}_\Lambda^{\mathrm{irr}}$ is irreducible. But the map $\pi_2\colon \widetilde{I}_\Lambda^{\mathrm{irr}} \to U$ is proper and dominant, and by Lemma \ref{3-quadrics-dim-estimate} its fibers are irreducible and equidimensional. Therefore $\widetilde{I}_\Lambda^{\mathrm{irr}}$ is irreducible, so $G_{[d]}$ is $4$-transitive. We conclude by Theorem \ref{CFSG} and Corollary \ref{NoMathieu} that $A_{N([d])}  \subset G_{[d]}$.
\end{proof}

\section{Intersections of two quadrics}\label{S:two-quadrics}
In this section we calculate the Fano monodromy group of the intersection of two quadrics, proving Theorem \ref{2-quadrics}.

We briefly describe the structure of the Fano scheme $F_k(Q_1 \cap Q_2)$ of a general intersection of two quadrics $Q_1, Q_2 \subset \PP^{2k+2}$, following \cite[Chapters~2~and~3]{Reid-thesis}. For the rest of the section, we assume that $\Char K \neq 2$. Suppose $Q_1, Q_2$ is a sufficiently general pair of quadratic forms. Then there exists a unique unordered $Q_1$-orthonormal basis $x_1, ..., x_{2k+3}\in H^0(\PP^{2k+3}, \calO(1))$ in which $Q_1$ and $Q_2$ are simultaneously diagonalized, so $Q_1$ and $Q_2$ are given by the equations $x_1^2+...+x_{2k+3}^2=0$ and $\lambda_1 x_1^2 + ... + \lambda_{2k+3} x_{2k+3}^2=0$ for some $\lambda_1, ..., \lambda_{2k+3} \in K$. Let $g_i \in \PGL_{2k+3}(K)$ denote the reflection along the hyperplane $x_i=0$. The group $G$ generated by $g_i$, for $i=1,..., 2k+3$ is isomorphic to $(\Z/2\Z)^{2k+2}$ and preserves $Q_1$ and  $Q_2$. The induced action of $G$ on $F_k(Q_1 \cap Q_2)$ is regular \cite[Theorem 3.14]{Reid-thesis}. In particular, $\deg F_k(Q_1 \cap Q_2)= 2^{2k+2}$ . By an \slantsf{incidence structure} of a collection of $k$-planes in $\PP^{2k+2}$, we mean the data of dimensions of their pairwise intersections. Fix a subspace $\Lambda \in F_k(Q_1 \cap Q_2)$. The incidence structure on $F_k(Q_1 \cap Q_2) = G \cdot \Lambda$ is determined by the following formula \cite[Theorem 3.14]{Reid-thesis}:
given an element $g=\prod_{i \in I} g_i = \prod_{i \notin I} g_i$ the intersection $\Lambda \cap g\Lambda$ has dimension $k-\min(\# I, 2k+3-\# I)$, with the convention $\dim \emptyset = -1$.

The Fano monodromy group has to preserve the incidence structure on $F_k(Q_1 \cap Q_2)$, which forces $G_{[d]}$ to be contained in the Weyl group $W(D_{2k+3})$ (alternatively: the lattice spanned by the classes of $k$-planes in the middle degree cohomology of $Q_1\cap Q_2$ can be proved to be of type $D_{2k+3}$; see \cite[Theorems~3.14~and~3.19]{Reid-thesis}). 
The group $W(D_{2k+3})$ is a semidirect product $\left(\Z/2\Z\right)^{2k+2} \rtimes S_{2k+3}$, where $S_{2k+3}$ acts on $\left(\Z/2\Z\right)^{2k+2}$ via the restriction of the permutation representation $S_{2k+3}\curvearrowright \F_2^{2k+3}$ to the sum-zero hyperplane. Let $\rho\colon W(D_{2k+3}) \to S_{2k+3}$ denote the projection. The surjection $\rho$ corresponds to the permutation of the coordinate hyperplanes $x_1, ..., x_{2k+3}$ induced by the symmetries of the incidence structure $F_k(Q_1 \cap Q_2).$ The Fano monodromy group $G_{[d]} = G_{(2,2)} \subset W(D_{2k+3})$ is transitive by Proposition \ref{transitive}. In Theorem \ref{Weyl.E.coyote} we show that $G_{[d]}=W(D_{2k+3})$ by showing that $\rho(G_{[d]})=S_{2k+3}$.

\begin{lemma}\label{2quad:surjectivity}
	Suppose $H \subset W(D_{2k+3})$ is a subgroup such that $\rho(H)=S_{2k+3}$ and $H$ acts transitively on $F_k(Q_1 \cap Q_2)$. Then  $H=W(D_{2k+3}).$
\end{lemma}
\begin{proof}
	Suppose $\ker \rho|_H = \{1 \}$, so $\rho\colon H \to S_{2k+3}$ is an isomorphism. Since $H$ acts transitively on a set of size $2^{2k+2}$, the number $2^{2k+2}$ should divide $\#S_{2k+3}$. But the maximal power of $2$ dividing $(2k+3)!$ is $k+1 + \lfloor\frac{k+1}{2}\rfloor+... < 2k+2$, contradiction. Therefore $\ker \rho|_H \neq \{1\}.$ Fix an element $1 \neq v \in H \cap \left(\Z/2\Z\right)^{2k+2}$ and choose a section $\phi$ of the map of sets $\rho\colon H \to S_{2k+3}$. Denote by $\cdot$ the action $S_{2k+3} \curvearrowright \left(\Z/2\Z\right)^{2k+2}$. For any $g \in S_{2k+3}$ the element $g \cdot v = \phi(g)v\phi(g)^{-1}$ belongs to $H$. Therefore $H$ contains the subgroup of $\left(\Z/2\Z\right)^{2k+2}$ generated by the orbit of a nontrivial element. A well-known fact, which we prove below, is that $\left(\Z/2\Z\right)^{2k+2}$ is an irreducible $S_{2k+3}$-module, and so $H$ contains the whole group $\left(\Z/2\Z\right)^{2k+2}.$
	
	We now show that the zero-sum hyperplane $V\colonequals (\Z/2\Z)^{2k+2}$ of the permutation module $(\Z/2\Z)^{2k+3}$ is an irreducible $S_{2k+3}$-module.  Consider a vector $0 \neq v=(v_1, \ldots, v_{2k+3}) \in V \subset (\Z/2\Z)^{2k+3}$. We need to show that $v$ generates $V$. Since $v$ is nonzero, there is a pair of indices $i,j$ with $v_i=0, v_j=1$. Then $v+(i,j)v$ is a vector with $i$ and $j$ coordinate equal to $1$ and the rest of coordinates equal to zero. Since $S_{2k+3}$ acts $2$-transitively, every vector that has exactly two nonzero coefficients is in the span of $S_{2k+3}v$. Finally, the vectors $(1,1,0,\ldots, 0), (0,1,1,0,\ldots, 0), (0,0,1,1,\ldots, 0), \ldots, (0,0,\ldots,1,1)$ form a basis of $V$. 
\end{proof}
\begin{theorem}\label{Weyl.E.coyote}
	The Fano monodromy group of a pair of quadrics is the full Weyl group, $G_{[d]}=G_{(2,2)}=W(D_{2k+3}).$
\end{theorem}
\begin{proof}
	Let $\Delta \subset M \times \PP^1$ be the subvariety of triples $(Q_1, Q_2, x)$ where $x$ is a root of the binary form $\det(uQ_1+vQ_2) \in K[u,v]$. The covering $\Delta \to M$ is generically \'etale and has Galois group $S_{2k+3}$ by \cite{Cohen}. Let $\Delta^{\Gal}$ denote the Galois closure of $\Delta \to M$; the scheme $\Delta^{\Gal}$ parameterizes pairs $Q_1, Q_2$ together with an ordering of the roots of $\det(uQ_1+vQ_2)$. We now relate this covering to the Fano covering $\pi$.
	
	Consider the scheme $I^{\mathrm{Gal}} \subset M \times \G(k, 2k+1)^{2^{2k+2}}$ parameterizing pairs of quadratic forms together with a collection of $2^{2k+2}$ distinct $k$-planes on $Q_1 \cap Q_2$. We want to show that the Galois group of $I^{\mathrm{Gal}} \to M$ is the full Weyl group $W(D_{2k+3})$ (in its regular action). For a sufficiently general pair $Q_1, Q_2$, an ordering of $k$-planes on $Q_1 \cap Q_2$ defines an ordering of the (a priori unordered) $Q_1$-orthonormal basis $x_1, ..., x_{2k+3} \in H^0(\PP^{2k+3},  \calO(1))$ in which $Q_1$ and $Q_2$ are simultaneously diagonalized as we now show. Consider the first $k$-plane $\Lambda$ and consider the ordering on the list of $k$-planes $\Lambda_1, ..., \Lambda_{2k+3}$ maximally incident to $\Lambda$. Since each $\Lambda_i$ is a reflection of $\Lambda$ with respect to a hyperplane $x_i=0$, the ordering $\Lambda_1, ..., \Lambda_{2k+3}$ defines an ordering of the basis $x_1, ..., x_{2k+3}$. Note that for all $i$ the $x_i^2$-coefficient of $Q_2$ is a root $\det(Q_1+xQ_2)$. Therefore there is a rational map $I^{\Gal} \dashrightarrow \Delta^{\Gal}$ such that after restricting to an open $U \subset M$ the covering $I^{\Gal}|_U \to U$  factors as $I^{\Gal}|_U \to \Delta^{\Gal}|_U \to U$. On the combinatorial side, the Weyl group is a semidirect product $1 \to (\Z/2\Z)^{2k+2} \to W(D_{2k+3}) \to S_{2k+3} \to 1$, where the surjection $\rho\colon W(D_{2k+3}) \to S_{2k+3}$ corresponds to the action of the Weyl group on the degenerate fibers of the pencil $uQ_1+vQ_2.$ Since the monodromy group of $\Delta^{Gal}\to M$ is the whole symmetric group, the group $G_{[d]}$ is a transitive subgroup of $W(D_{2k+3})$ such that $\rho$ restricted to $G_{[d]}$ is a surjection. By Lemma \ref{2quad:surjectivity} we have $G_{[d]}=W(D_{2k+3}).$
\end{proof}

\section*{Acknowledgments} 
We thank Brendan Hassett for suggesting the problem and for his remarks on Theorem \ref{Weyl.E.coyote}. We are grateful to Hannah Larson for her numerous comments on an earlier draft of the paper. Finally, we thank Bjorn Poonen for his advice on the computation in Lemma \ref{38406501359372282063949}.

\normalsize
\clearpage
\bibliographystyle{alpha}
\bibliography{bibliography}	
\end{document}